\documentclass[a4paper]{article}
\usepackage{lmodern}

\usepackage{stmaryrd,amsmath,amssymb}
\usepackage{amsthm}
\usepackage[mathjax]{lwarp}

\usepackage{hyperref}

\newcommand{\mathmacro}[1]{#1\CustomizeMathJax{#1}}

\mathmacro{
\newcommand{\GL}{\operatorname{GL}}
\newcommand{\PGL}{\operatorname{PGL}}
\newcommand{\SL}{\operatorname{SL}}
\newcommand{\PSL}{\operatorname{PSL}}
\newcommand{\normal}{\trianglelefteqslant}
\newcommand{\gen}[1]{\langle#1\rangle}
\newcommand{\ep}{\varepsilon}
\newcommand{\F}{\mathbb{F}}
}
\newtheorem{prop}{Proposition}
\newtheorem{thm}[prop]{Theorem}
\newtheorem{cor}[prop]{Corollary}
\newtheorem{lem}[prop]{Lemma}
\newtheorem{question}[prop]{Question}

\theoremstyle{definition}

\newtheorem{defn}[prop]{Definition}
\newtheorem{example}[prop]{Example}

\theoremstyle{remark}
\newtheorem{rem}[prop]{Remark}

\numberwithin{prop}{section}

\title{The Ingleton inequality holds for metacyclic groups and fails for supersoluble groups}
\author{David A.~Craven}
\date{\today}

\begin{document}
\maketitle

\section{Introduction}

The Ingleton inequality first appeared in matroid theory, where Ingleton proved in 1971 \cite{ingleton1971} that every rank function coming from a representable matroid on four subsets satisfies a particular inequality. Because this inequality is not implied by submodularity, Shannon-type axioms alone, it and various analogues  play a central role in separately linear and non-linear phenomena in a variety of areas of mathematics.

One particularly important area is information theory, where the Shannon entropy of discrete random variables obeys the submodularity relations that matroid rank functions obey. The Ingleton inequality concerns network codes in four random variables, and the part of the (Shannon) entropy region where the Ingleton inequality holds, corresponding to linear network codes, forms a polyhedron. However, there are points in the entropy region where Ingleton's inequality is violated, i.e., that lie outside this polyhedron. (See for example \cite{bostonnan2020} and the references therein for more details.)

It is possible to study these points just using finite group theory, and this is the approach taken in this paper. The statement of the Ingleton inequality is easy, and depends only on four subgroups, $H_1$, $H_2$, $H_3$ and $H_4$ of a finite group $G$.

Formally, let $G$ be a finite group and let $\mathbf H=(H_1,H_2,H_3,H_4)$ be a quadruple of subgroups of $G$. Write $H_{12}$ for the intersection $H_1\cap H_2$ and so on. We call $\mathbf H$ an \emph{Ingleton offender} if the inequality 
\[ |H_1|\cdot|H_2|\cdot|H_{34}|\cdot|H_{123}|\cdot|H_{124}|\geq |H_{12}|\cdot|H_{13}|\cdot|H_{14}|\cdot|H_{23}|\cdot|H_{24}|\]
is not satisfied by the $H_i$. We are normally interested in the case where $G$ is generated by the $H_i$, which we call a \emph{generative Ingleton offender}. If $G$ contains an Ingleton offender, we call $G$ an \emph{Ingleton violator}, and similarly a \emph{generative Ingleton violator} if it contains a generative Ingleton offender.

There are infinitely many generative Ingleton violators, with the first examples discovered being $\PGL_2(p)$ for all $p\geq 5$ \cite{maohassibi2009}. Since then, soluble Ingleton violators have been found \cite{bostonnan2020}, most interestingly $A_4\times A_4$, which shows that the direct product of two groups can be an Ingleton violator despite its two factors not being so. On the other hand, there are no abelian Ingleton violators \cite{chan2007}. It was claimed in \cite{oggierstancu2012} that there were no metacyclic or nilpotent Ingleton violators, but the proof contains an error, which we discuss in Section \ref{sec:primeorder}. We can, however, still prove one of these using a different approach.

\begin{thm}\label{thm:nometacyclic} There are no metacyclic Ingleton violators.
\end{thm}

On the other hand, we do have the following result, which closes the gap still further in the negative direction. Recall that a group is \emph{supersoluble} (or supersolvable) if it has a chain of subgroups
\[ 1=G_0\leq G_1\leq \cdots \leq G_r=G\]
where each $G_i$ is normal in $G$ and $G_i/G_{i-1}$ is cyclic. Every nilpotent group is supersoluble, but there are others, like $S_3$ for example. Notice that $A_4$ is not supersoluble, so supersoluble is quite close to nilpotent.

\begin{thm} There are infinitely many supersoluble generative Ingleton violators, specifically semidirect products of extraspecial groups $p^{1+2}_+$ by a cyclic group of order $p-1$.
\end{thm}

We cannot prove that there are no nilpotent Ingleton violators in this article, but we do give some tight restrictions on the structure of a nilpotent Ingleton violator of minimal order in Section \ref{sec:primeorder}. We prove the following result.

\begin{thm}\label{thm:removenormal} Let $G$ be a finite group and suppose that $\mathbf H=(H_1,H_2,H_3,H_4)$ is an Ingleton offender in $G$. If $N$ is a normal subgroup of $G$ then either $H_i\cap N=1$ for all $1\leq i\leq 4$, or the images of the $H_i$ in $G/N$ form an Ingleton offender, i.e.,  $(H_1N/N,H_2N/N,H_3N/N,H_4N/N)$ is an Ingleton offender in $G/N$.
\end{thm}

Note that the existence of supersoluble Ingleton violators implies that the option $H_i\cap N=1$ in the theorem above can, and does, occur. Because of this result, we say that a generative Ingleton offender $(H_1,H_2,H_3,H_4)$ in a finite group $G$ is \emph{irreducible} if any normal subgroup of $G$ intersects all $H_i$ trivially. If $N$ is a normal subgroup of a finite group $G$ and $G/N$ possesses an Ingleton offender, then by taking full preimages of the subgroups we obtain an Ingleton offender in $G$. There might be others as well, including irreducible ones, that nevertheless still have image an Ingleton offender in $G/N$. We therefore restrict things further and call an Ingleton offender \emph{indomitable} if the image of it modulo any non-trivial normal subgroup is not an Ingleton offender. There are examples of irreducible Ingleton offenders that are not indomitable, see Section \ref{sec:smallorder}. Of course, if $G/N$ is not an Ingleton violator for each $1\neq N\normal G$ then every generative Ingleton offender in $G$ is indomitable. A minimal example of a nilpotent Ingleton violator would possess only indomitable Ingleton offenders.

\medskip

The final section of this article contains an exhaustive computer search, which finds all irreducible Ingleton offenders in finite groups of order less than $1023$. (This includes checking that there are no Ingleton violators of order $512$.) 

\begin{thm} If $|G|<1024$ and $G$ is an Ingleton violator, then $G$ is known.
\end{thm}

As a result of this search, we can ask another question, beyond the potential existence of nilpotent Ingleton violators.

\begin{question} Is there an Ingleton violator of odd order?
\end{question}

There are Ingleton violators with cyclic Sylow $2$-subgroup, including groups with twice odd order, among the supersoluble groups in Section \ref{sec:supersoluble}, but no group of odd order so far. Indeed, there are no odd-order Ingleton violators of order up to $2000$, by an exhaustive computer search. The prime $2$ seems to play an important role, but one that is not yet clear.

\medskip

The structure of the paper is as follows. Section \ref{sec:prelims} collects some preliminary lemmas, and proves a general condition: if $H_{12}=H_{123}H_{124}$ then the Ingleton inequality always holds for $(H_1,H_2,H_3,H_4)$. This is used to prove that $H_1$ and $H_2$ are non-cyclic in any Ingleton offender. In this section we also prove the known general conditions for a quadruple to not be an Ingleton offender, which are scattered in the literature and sometimes in different language to finite group theory. Section \ref{sec:primeorder} proves Theorem \ref{thm:removenormal}, and together with the fact that $H_1$ cannot be cyclic in an Ingleton offender, proves Theorem \ref{thm:nometacyclic}. Section \ref{sec:supersoluble} constructs the supersoluble Ingleton violators, and Section \ref{sec:smallorder} gives details about all groups of order at most $1023$ that posses irreducible and indomitable Ingleton offenders.

\medskip

The proof here that metacyclic groups are not Ingleton violators was created more than a decade ago, but I did not publish it as I thought it was already known. I would like to thank Fr\'ed\'erique Oggier for alerting me to the fact that the original proof in \cite{oggierstancu2012} has a gap. I would like to thank Chris Parker for suggesting to look at all primes $p$ when I showed him the supersoluble example $5^{1+2}_+\rtimes 4$, which led to the results in Section \ref{sec:supersoluble}.

The calculations here were all made using Magma \cite{magma} and the author includes full code and a database of Ingleton offenders as auxiliary materials with this paper.

\section{Preliminary lemmas and general conditions}
\label{sec:prelims}

In this section we collect together various results on intersections of subgroups that we will need to analyse how the intersection of two subgroups changes as you modify one of the two subgroups. We then move on to proving some general statements about Ingleton offenders, many of which have appeared in the literature before but scattered, and some with only proofs using information-theoretic methods rather than group theory. We end by proving that in an Ingleton offender, neither $H_1$ nor $H_2$ can be cyclic, a crucial tool for proving Theorem \ref{thm:nometacyclic} and for case reduction when engaging in a computer search.

The order formula $|AB|=|A|\cdot |B|/|A\cap B|$ immediately gives us the first lemma.

\begin{lem}\label{lem:orderintersection} If $A$ and $B$ are subgroups of a finite group $G$ then $|G|\cdot |A\cap B|\geq |A|\cdot |B|$.
\end{lem}

We now give a few lemmas controlling intersections, and considering intersections of the form $AB\cap C$.

\begin{lem}\label{lem:indexreduction} Let $A$, $B$ and $C$ be subgroups of a group $G$, with $A\geq B$. We have that $|A:A\cap C|\geq |B:B\cap C|$, or equivalently $|A|/|B|\geq |A\cap C|/|B\cap C|$.
\end{lem}
\begin{proof} Since $A\geq B$ we have that $|AC|\geq |BC|$, so certainly $|AC|/|C|\geq |BC|/|C|$. But from the order formula this is the same thing as $|A|/|A\cap C|\geq |B|/|B\cap C|$, as claimed.
\end{proof}


\begin{lem}\label{lem:maxlp} Let $A$ and $B$ be subgroups of $G$, and let $N$ be a normal subgroup of $G$. We have that $|AN\cap B:A\cap B|\leq |N:N\cap A|$, with equality if $N\leq B$.
\end{lem}
\begin{proof} Without loss of generality we may assume that $G=AN$, since we may replace $B$ by $AN\cap B$. Note that $(AN\cap B)\cap A=A\cap B$, and so by Lemma \ref{lem:orderintersection}, $|AN|\cdot |A\cap B|\geq |AN\cap B|\cdot |A|$. Rearranging, we see that $|AN:A|\geq |AN\cap B:A\cap B|$. But $N$ is a normal subgroup, so $|AN:A|=|N:N\cap A|$ which divides $n$. Thus the result holds in general. If $N\leq B$ then $A(AN\cap B)$ contains $A$ and $N$, so $A(AN\cap B)=G$, so we obtain an equality.
\end{proof}

We now find an apparently new criterion for the Ingleton inequality to always hold: if $H_{12}=H_{123}H_{124}$.

When working with the Ingleton inequality, we will need to reference the terms on either side of it. On the left-hand side, call (a) to (e) the terms in the product: (a) is $|H_1|$, (b) is $|H_2|$, (c) is $|H_{34}|$, (d) is $|H_{123}|$, and (e) is $|H_{124}|$. On the right-hand side, use ($\alpha$) to ($\ep$): ($\alpha$) is $|H_{12}|$, ($\beta$) is $|H_{13}|$, ($\gamma$) is $|H_{14}|$, ($\delta$) is $|H_{23}|$ and ($\ep$) is $|H_{24}|$. This makes it easy to see how the two sides of the inequality balance, because we use Roman and Greek letters.

\begin{lem}\label{lem:reducesizeofH1} Let $G$ be a finite group and let $\mathbf H=(H_1,H_2,H_3,H_4)$ be a quadruple of subgroups of $G$. If $(H_1,H_2,H_3,H_4)$ is an Ingleton offender then so is $\bar{\mathbf H}=(\gen{H_{13},H_{14}},H_2,H_3,H_4)$.
\end{lem}
\begin{proof} Write $\bar H_1$ for $\gen{H_{13},H_{14}}$. We calculate the effect of this replacement on the Ingleton inequality. Clearly $H_1\cap H_3=\bar H_1\cap H_3$, so the only affected quantities are $|H_1|$ and $|H_{12}|$, i.e., (a) and ($\alpha$). If $\mathbf H$ is an Ingleton offender already, and the reduction in the left-hand side is at least as large as the reduction in the right-hand side, then $\bar{\mathbf{H}}$ will still be an Ingleton offender. But the change in the left-hand side is $|H_1|/|H_1\cap \bar H_1|$, and the change in the right-hand side is $|H_{12}|/|H_{12}\cap \bar H_1|$. By Lemma \ref{lem:indexreduction}, this means that the left-hand side of the Ingleton inequality drops by at least as much as the right-hand side, so if $\textbf H$ is an Ingleton offender so is $\bar{\textbf H}$, as claimed.
\end{proof}

\begin{rem} This is the specific case that we need in Corollary \ref{cor:H1notcyclic}, but the exact same proof shows that, given an Ingleton offender $(H_1,H_2,H_3,H_4)$, we can always find Ingleton offenders contained in that quadruple until we hit one where
\[ H_1=\gen{H_{12},H_{13}}=\gen{H_{12},H_{14}}=\gen{H_{13},H_{14}},\]
\[ H_2=\gen{H_{12},H_{23}}=\gen{H_{12},H_{24}}=\gen{H_{23},H_{24}},\]
\[ H_3=\gen{H_{13},H_{23}},\quad H_4=\gen{H_{14},H_{24}}.\]
Coupled with certain of the $H_i$ being abelian, this could help exclude certain cases. Note that there is no known Ingleton offender with $H_1$ or $H_2$ abelian, or with $H_3$ or $H_4$ cyclic. However, as found in \cite{bostonnan2020}, the Ingleton offender for $A_4\times A_4$ has $H_3$ and $H_4$ abelian, isomorphic to $C_3\times C_3$.

We prove that $H_1$ and $H_2$ cannot be cyclic in Corollary \ref{cor:H1notcyclic} below. Although the author does not see how to prove in general that $H_3$ cannot be cyclic, it can be seen that it is not a cyclic group of prime-power order. For if it were, then in a minimal example where $H_3=\gen{H_{13},H_{23}}=H_{13}H_{23}$, one of $H_{13}$ and $H_{23}$ contains the other, without loss of generality $H_{13}\leq H_{23}$. But then $H_3=H_{23}$ and we cannot have $H_3\leq H_2$ by Lemma \ref{lem:generalexclusions} below. This is potentially of value when searching for nilpotent Ingleton offenders.
\end{rem}

\begin{prop}\label{prop:ingoffenderfact} If $\mathbf H=(H_1,H_2,H_3,H_4)$ is a quadruple of subgroups of a finite group $G$ with $H_{12}=H_{123}H_{124}$ then $\mathbf H$ is not an Ingleton offender.
\end{prop}
\begin{proof} By Lemma \ref{lem:orderintersection}, $|H_1|\cdot |H_{134}|\geq |H_{13}|\cdot |H_{14}|$, since $|H_1|\geq |H_{13}H_{14}|$. Similarly, $|H_2|\cdot |H_{234}|\geq |H_{23}|\cdot |H_{24}|$ and |$H_{34}|\cdot |H_{1234}|\geq |H_{134}|\cdot |H_{234}|$.
Combining these three inequalities yields
\[ |H_1|\cdot |H_2|\cdot |H_{134}|\cdot |H_{234}|\cdot |H_{34}|\cdot |H_{1234}|\geq |H_{13}|\cdot |H_{14}|\cdot |H_{23}|\cdot |H_{24}|\cdot |H_{134}|\cdot |H_{234}|.\]
Cancelling off the identical terms on both sides yields
\[ |H_1|\cdot |H_2|\cdot |H_{34}|\cdot |H_{1234}|\geq |H_{13}|\cdot |H_{14}|\cdot |H_{23}|\cdot |H_{24}|.\]
This holds for all $\mathbf{H}$. If $H_{12}=H_{123}H_{124}$, we have the equality $|H_{12}|\cdot |H_{1234}|=|H_{123}|\cdot |H_{124}|$, which we can swap around and combine with the inequality to yield. 
\[ |H_1|\cdot |H_2|\cdot |H_{34}|\cdot |H_{1234}|\cdot |H_{123}|\cdot |H_{124}|\geq |H_{13}|\cdot |H_{14}|\cdot |H_{23}|\cdot |H_{24}|\cdot |H_{12}|\cdot |H_{1234}|.\]
Cancelling off the $|H_{1234}|$ terms yields the Ingleton inequality.
\end{proof}

\begin{cor}\label{cor:H1notcyclic} If $(H_1,H_2,H_3,H_4)$ is an Ingleton offender then $H_1$ and $H_2$ cannot be cyclic.
\end{cor}
\begin{proof} By interchanging $H_1$ and $H_2$ if necessary, assume that $H_1$ is cyclic. Applying Lemma \ref{lem:reducesizeofH1}, if there is an example with $H_1$ cyclic then we can also assume that $H_1=H_{13}H_{14}$. We notice that if $X=AB$ is cyclic and $C\leq X$ then $C=(A\cap C)(B\cap C)$. (This follows easily from the fact that, for cyclic groups, a subgroup is determined by its order, and $|A\cap B|=\gcd(|A|,|B|)$ in cyclic groups.) Since $H_1$ is cyclic, $H_{12}=H_1\cap H_2=(H_{13}\cap H_{12})(H_{14}\cap H_{12})=H_{123}H_{124}$, and so we cannot be an Ingleton offender by Proposition \ref{prop:ingoffenderfact}.
\end{proof}

The end this section by proving a number of general criteria that force a given quadruple not to be an Ingleton offender. These conditions were mostly previously known, but we collect them here, with full, group-theoretic proofs, for the reader's convenience.

Note that the Ingleton bound is invariant under swapping $H_1$ and $H_2$, and under swapping $H_3$ and $H_4$.

\begin{lem}\label{lem:generalexclusions} Let $\mathbf{H}=(H_1,H_2,H_3,H_4)$ be a quadruple of subgroups of a finite group $G$. If any of the following hold, then $\mathbf{H}$ is not an Ingleton offender.
\begin{enumerate}
\item Any of $H_{12}$, $H_{13}$, $H_{14}$, $H_{23}$ and $H_{24}$ being trivial.
\item Any $H_i$ is contained in any $H_j$ for $i\neq j$.
\item $H_1H_2$ is a subgroup of $G$.
\end{enumerate}
Consequently, if $H_1$ or $H_2$ is a normal subgroup of $G$, or if $G$ is abelian, then $\mathbf H$ is not an Ingleton violator.
\end{lem}
\begin{proof} Recall that we may swap $H_1$ and $H_2$, and swap $H_3$ and $H_4$, without affecting the Ingleton inequality. Thus we need only show that $H_{12}=1$ and $H_{13}=1$ forces $\mathbf H$ to satisfy the Ingleton inequality. Proposition \ref{prop:ingoffenderfact} proves that $\mathbf H$ cannot be an Ingleton offender if $H_{12}=1$, or if $H_3\geq H_1$ (for then $H_{123}=H_{12}$).

If $H_{13}=1$ then so is $H_{123}$, and the Ingleton inequality reduces to
\[ |H_1|\cdot|H_2|\cdot|H_{34}|\cdot |H_{124}|\geq |H_{12}|\cdot|H_{14}|\cdot|H_{23}|\cdot|H_{24}|.\]
We certainly have $|H_1|\cdot |H_{124}|\geq |H_{12}|\cdot |H_{14}|$ and
\[|H_2|\cdot |H_{34}|\geq |H_2|\cdot |H_{234}|\geq |H_{23}|\cdot |H_{24}|,\]
so the result holds.

For the containment statement, by moving around the $H_i$ we need check $H_1\geq H_2$, $H_3\geq H_4$ and $H_1\geq H_3$ (since we have already seen the case $H_3\geq H_1$).
\begin{itemize}
\item If $H_1\geq H_2$ then the Ingleton inequality reduces to
\[ |H_1|\cdot|H_2|\cdot|H_{34}|\cdot|H_{23}|\cdot|H_{24}|\geq |H_{2}|\cdot|H_{13}|\cdot|H_{14}|\cdot|H_{23}|\cdot|H_{24}|.\]
Cancelling off leaves $|H_1|\cdot |H_{34}|\geq |H_{13}|\cdot |H_{14}|$, and this again holds because $|H_{34}|\geq |H_{134}|$.

\item If $H_3\geq H_4$ then $H_{34}=H_4$, and we use $|H_4|\cdot |H_{124}|\geq |H_{14}|\cdot |H_{24}|$ and $|H_1|\cdot |H_{123}|\geq |H_{12}|\cdot |H_{13}|$. Removing these from the Ingleton inequality leaves us left to prove $|H_2|\geq |H_{23}|$, which is obvious.

\item If $H_1\geq H_3$ then the inequality reduces to
\[ |H_1|\cdot|H_2|\cdot|H_{34}|\cdot|H_{23}|\cdot|H_{124}|\geq |H_{12}|\cdot|H_{3}|\cdot|H_{14}|\cdot|H_{23}|\cdot|H_{24}|.\]
By Lemma \ref{lem:indexreduction} we have $|H_1:H_{14}|\geq |H_3:H_{34}|$, so $|H_1|\cdot |H_{34}|\geq |H_3|\cdot |H_{14}|$. Thus if $|H_2|\cdot |H_{124}|\geq |H_{12}|\cdot |H_{24}|$ then we are done, but this is Lemma \ref{lem:orderintersection} again.
\end{itemize}

We now deal with the case where $H_1H_2\leq G$. If $H_1H_2$ is a subgroup of $G$ then $|H_1H_2|=|H_1|\cdot |H_2|/|H_{12}|$, so the Ingleton inequality becomes
\[ |H_1H_2|\cdot|H_{34}|\cdot|H_{123}|\cdot|H_{124}|\geq |H_{13}|\cdot|H_{14}|\cdot|H_{23}|\cdot|H_{24}|.\]
Now, by the order formula again, the \emph{sets} (not necessarily subgroups) $H_{13}H_{23}$ and $H_{14}H_{24}$ have orders $|H_{13}|\cdot|H_{23}|/|H_{123}|$ and $|H_{14}|\cdot |H_{24}|/|H_{124}|$ respectively, so substituting these in yields
\[ |H_1H_2|\cdot|H_{34}|\geq |H_{13}H_{23}|\cdot|H_{14}H_{24}|.\]
Using the order formula again yields
\[ |H_1H_2|\cdot|H_{34}|\geq |H_{13}H_{23}H_{14}H_{24}|\cdot |(H_{13}H_{23})\cap(H_{14}H_{24})|.\]
However, the former of the sets on the right-hand side lies in $H_1H_2$ since $H_1H_2$ is a group, and the latter lies in $H_{34}$, so the inequality holds.
\end{proof}

Note that $H_{34}=1$ is definitely allowed, since $H_{34}=1$ is the case for a number of the known Ingleton offenders, for example the supersoluble ones in Section \ref{sec:supersoluble}. Those examples also have the property that $H_{12}$ has prime order and that $H_{ij}$ is cyclic for all $i,j$. Thus we cannot place more restrictions on the intersections than that they are non-trivial.

\section{Normal subgroups and Ingleton offenders}
\label{sec:primeorder}

If $N$ is a normal subgroup of a finite group $G$, and $G/N$ possesses an Ingleton offender, then by taking full preimages we must obtain an Ingleton offender for $G$, because all subgroup orders are just multiplied by $|N|$. Thus if $H_{1234}$ contains a normal subgroup $N$ of $G$, then we can simply quotient out by it to obtain an Ingleton offender for $G/N$. These Ingleton offenders are not particularly interesting.

In this section we will explore to what extent we can have a normal subgroup $N$ that is a subgroup of \emph{some}, but not necessarily \emph{all}, of the $H_i$ in an Ingleton offender. We will prove that if $N$ is a normal subgroup of $G$ and $(H_1,H_2,H_3,H_4)$ is an Ingleton offender with $N\leq H_i$ for some $i$, then $(NH_1,NH_2,NH_3,NH_4)$ is also an Ingleton offender. This will prove Theorem \ref{thm:removenormal}.

\begin{proof}[Proof of Theorem \ref{thm:removenormal}] Let $N$ be a normal subgroup contained in one of the $H_i$. By permuting either $H_1$ and $H_2$, or $H_3$ and $H_4$ if necessary, we may assume that $N\leq H_2$ or $N\leq H_4$.

First, if $N\leq H_i$ for all $i$, then write $\bar H_i$ for $H_i/N$, and similarly for the intersections. Then $|\bar H_I|=|H_I|/|N|$ for all subsets $I$ and so Ingleton's inequality fails for the $H_i$ if and only if it fails for the $\bar H_i$. Thus the result holds if $N\leq H_{1234}$.

Hence we may assume that $N\not\leq H_1$ or $N\not\leq H_3$ (or both). Assume first that $N\not\leq H_1$, and let $J_1=NH_1$, $J_i=H_i$ for $i=2,3,4$. We claim that the $J_i$ form an Ingleton offender in $G$. To see this, we track the change to the Ingleton inequality caused by replacing the $H_i$ by the $J_i$. Note that $|H_I|=|J_I|$ if $1\not\in I$, and $|J_1|=|H_1|\cdot (|N|/|N\cap H_1|)$, so we only need to consider the change for the $|H_I|$ for $1\in I$. By assumption $N\leq H_i$ for $i=2$ or $i=4$, and in this case, by Lemma \ref{lem:maxlp},
\[ \frac{|J_1\cap J_i|}{|H_1\cap H_i|}=\frac{|H_1N\cap H_i|}{|H_1\cap H_i|}=\frac{|N|}{|N\cap H_1|}.\] Thus the change in the contribution to the inequality for $|H_1|$ and $|H_{1i}|$ are the same, so these can be ignored as well. On the left-hand side, we are left with $H_{123}$ and $H_{124}$, and on the right-hand side we are left with $H_{13}$ and either $H_{12}$ or $H_{14}$. But by Lemma \ref{lem:indexreduction}
\[ \frac{|J_{123}|}{|H_{123}|}=\frac{|J_{13}\cap H_2|}{|H_{13}\cap H_2|}\leq \frac{|J_{13}|}{|H_{13}|}.\]
Similarly $|J_{124}|/|H_{124}|$ is at most $|J_{12}|/|H_{12}|$ and at most $|J_{14}|/|H_{14}|$, so the left-hand side of Ingleton's inequality increases by at most that of the right-hand side upon replacing the $H_i$ by the $J_i$. This proves that the $J_i$ also form an Ingleton offender.

Thus we may assume that $N\leq H_{12}$ but $N\not\leq H_3$. This time set $J_3=NH_3$ and $J_i=H_i$ for $i=1,2,4$. Then the only quantities in Ingleton's inequality that are affected are $|H_{34}|$ and $|H_{123}|$ on the left-hand side, and $|H_{13}|$ and $|H_{23}|$ on the right-hand side. Since $N\leq H_1$, by Lemma \ref{lem:maxlp} (as in the previous case) we have
\[ \frac{|J_{13}|}{|H_{13}|}=\frac{|J_{23}|}{|H_{23}|}=\frac{|J_{123}|}{|H_{123}|}=\frac{|N|}{|N\cap H_3|}.\]
On the other hand, $|J_{34}|/|H_{34}|=|NH_3\cap H_4|/|H_3\cap H_4|\leq |N|/|N\cap H_3|$ by Lemma \ref{lem:maxlp} again, so again the increase in the left-hand side of Ingleton's inequality is at most the increase in the right-hand side. Thus the $J_i$ again form an Ingleton offender.

Thus we have proved that if $N\leq H_i$ for any $i$ then we can increase the number of the $H_i$ containing $N$ and remain being an Ingleton offender. If all $H_i$ contain $N$ then the $H_i/N$ also form an Ingleton offender. This proves the result.
\end{proof}

Using this, we may prove that metacyclic groups are not Ingleton violators.

\begin{proof}[Proof of Theorem \ref{thm:nometacyclic}] Let $G$ be a metacyclic group with normal subgroup $K$ such that $K$ and $G/K$ are cyclic, and proceed by induction on $|G|$. If $\mathbf H$ is an Ingleton offender then $H_1$ cannot be cyclic by Corollary \ref{cor:H1notcyclic}. As $G/K$ is cyclic, $H_1\cap K$ cannot be trivial, whence $H_1$ contains some non-trivial normal subgroup $N$ of $K$, necessarily normal in $G$. But $G/N$ is not an Ingleton violator by induction, so $H_1\cap N=1$ by the previous case analysis. This is a contradiction, so $\mathbf H$ does not exist.
\end{proof}

\begin{rem}In the remaining case, where $N$ intersects all $H_i$ trivially, one cannot do the same analysis. Recall our labelling the terms of Ingleton's inequality by (a) to (e) and ($\alpha$) to ($\ep$). If one writes $\bar H_1=NH_1$ and $\bar H_i=H_i$ for $i=2,3,4$, then (a) increases by a factor of $p$, but (d), (e), ($\alpha$), ($\beta$) and ($\gamma$) all increase by a factor of $1$ or $p$. If (d) or (e) does increase by a factor of $p$ then everything works, but it could be that $\bar H_{123}=H_{123}$ and $\bar H_{124}=H_{124}$.

If one attempts to proceed by setting $\bar H_3=NH_3$ and $\bar H_i=H_i$ for $i=1,2,4$, then one sees that (c), (d), ($\beta$) and ($\delta$) are affected. But a change in $|H_{34}|$ does not, \emph{a priori}, affect either $H_{13}$ or $H_{23}$. If $\bar H_{123}>H_{123}$ then the proof goes through, as with Cases 3 onwards above. So the problem here is that we could have $\bar H_{34}>H_{34}$ while $\bar H_{123}=H_{123}$. Thus in order for the images of the $H_i$ not to form an Ingleton offender, we need that $N\not\leq H_iH_j$ for $\{i,j\}\neq \{3,4\}$ and $N\leq H_3H_4$.

Indeed, this is not a theoretical problem. The supersoluble groups $G$ in Section \ref{sec:supersoluble} have a normal subgroup $N$ of order $p$, and the images of the subgroups in $G/N$ do not form an Ingleton offender. So this method of proof, as attempted in \cite{oggierstancu2012}, cannot obviously prove that nilpotent groups are not Ingleton violators.

One might suggest there is some benefit to the normal subgroup being central rather than just normal, as in the supersoluble case. However, the group $3\times \PSL_2(7)$ is also an Ingleton violator, as we see in Section \ref{sec:smallorder}, whereas no proper subgroup or quotient of that group is an Ingleton violator, so forcing centrality does not help.
\end{rem}

Having proved that, if you have a normal subgroup contained in any $H_i$ then you are just a (not necessarily full) preimage of an Ingleton offender for $G/N$, it makes sense to make the following definition.

\begin{defn} An Ingleton offender $(H_1,H_2,H_3,H_4)$ in a finite group $G$ is \emph{irreducible} if it is a generative Ingleton offender and no $H_i$ contains a non-trivial normal subgroup of $G$. An irreducible Ingleton offender is \emph{indomitable} if, given any non-trivial normal subgroup of $G$, the image of $(H_1,H_2,H_3,H_4)$ in $G/N$ is not an Ingleton offender. We extend our terminology to \emph{irreducible Ingleton violator} and \emph{indomitable Ingleton violator} in the obvious way.
\end{defn}

If trying to prove that groups with a certain property cannot be Ingleton violators, one will only need to search for indomitable offenders. Given full knowledge of the generative Ingleton offenders in proper quotient groups of a finite group $G$, one can construct all generative Ingleton offenders other than the indomitable ones by taking full preimages of the $\bar H_i$ in all Ingleton offenders in $G/N$ and then enumerating all subgroups $H_i$ such that $NH_i/N=\bar H_i$. Although this could be time-consuming, it is significantly better than searching for indomitable Ingleton offenders, where the only known method is to enumerate all subgroups of a group and calculate intersections.

\section{Supersoluble counterexamples}
\label{sec:supersoluble}
We consider a specific subgroup of the Borel subgroup of $\SL_3(p)$: let
\[ u_1=\begin{pmatrix}1&1&0\\0&1&0\\0&0&1\end{pmatrix},\quad u_4=\begin{pmatrix}1&0&0\\0&1&1\\0&0&1\end{pmatrix},\quad t=\begin{pmatrix}\zeta&0&0\\0&1&0\\0&0&\zeta^{-1}\end{pmatrix},\]
where $\zeta$ is a primitive element of $\F_p$. The group $G=\gen{u_1,u_4,t}$ is a supersoluble group of shape $p^{1+2}_+\rtimes (p-1)$. We claim that it is an Ingleton violator for any $p\geq 5$.

We introduce some more elements of $G$: let
\[  u_2=\begin{pmatrix}1&1&\zeta(\zeta-2)/(1-\zeta)\\0&1&-2\zeta\\0&0&1\end{pmatrix}\quad u_3=\begin{pmatrix}1&0&1/(1-\zeta)\\0&1&1\\0&0&1\end{pmatrix},\]
\[ h_2=\begin{pmatrix}\zeta&0&0\\0&1&1\\0&0&\zeta^{-1}\end{pmatrix},\quad h_3=\begin{pmatrix}\zeta&1&0\\0&1&0\\0&0&\zeta^{-1}\end{pmatrix},\]
write $H_1=\gen{u_1,t}$, $H_2=\gen{u_2,h_2}$, $H_3=\gen{u_3,h_3}$ and $H_4=\gen{u_4,t}$.

We first claim that all $H_i$ are Frobenius groups of order $p(p-1)$.

\begin{lem} We have that $t$ normalizes $\gen{u_1}$ and $\gen{u_4}$, that $h_2$ normalizes $\gen{u_2}$ and that $h_3$ normalizes $\gen{u_3}$.
\end{lem}
\begin{proof} The subgroup generated by $t$ and $u_1$, and $t$ and $u_4$, is a Borel subgroup of an $\SL_2$-Levi subgroup of $\SL_3(p)$, so the result holds for these elements.

We next prove that $h_3$ normalizes $\gen{u_3}$. To see this, we replace the $(1,3)$-entry by an arbitrary $x$ and conjugate by $h_3$:
\[ \begin{pmatrix}\zeta^{-1}&-\zeta^{-1}&0\\0&1&0\\0&0&\zeta\end{pmatrix}\begin{pmatrix}1&0&x\\0&1&1\\0&0&1\end{pmatrix}\begin{pmatrix}\zeta&1&0\\0&1&0\\0&0&\zeta^{-1}\end{pmatrix}=\begin{pmatrix}1&0&\zeta^{-2}(x-1)\\0&1&\zeta^{-1}\\0&0&1\end{pmatrix}.\]
To lie in the subgroup generated by the middle matrix in the conjugation, we must have that the ratio $\zeta^{-2}(x-1)/\zeta^{-1}$ is equal to $x$. Rearranging, we obtain $x=1/(1-\zeta)$ as the unique possible value. Thus $H_3$ is a Frobenius group of order $p(p-1)$.

Finally, we conjugate the matrix $u_3$ by the matrix $h_2$:
\[ \begin{pmatrix}\zeta^{-1}&0&0\\0&1&-\zeta\\0&0&\zeta\end{pmatrix}\begin{pmatrix}1&1&\zeta(\zeta-2)/(1-\zeta)\\0&1&-2\zeta\\0&0&1\end{pmatrix}\begin{pmatrix}\zeta&0&0\\0&1&1\\0&0&\zeta^{-1}\end{pmatrix}=\begin{pmatrix}1&\zeta^{-1}&1/\zeta(\zeta-1)\\0&1&-2\\0&0&1\end{pmatrix}.\]
It is easy to see that
\[ u_2^i=\begin{pmatrix}1&i&i\alpha-i(i-1)\zeta\\0&1&-2i\zeta\\0&0&1\end{pmatrix},\]
where $\alpha=\zeta(\zeta-2)/(1-\zeta)$ is the $(1,3)$-entry of $u_2$. We just need to check that the conjugate is the $\zeta^{-1}$th power of $u_2$. But substituting $i=\zeta^{-1}$ indeed yields the conjugate matrix, so $h_2$ does normalize $\gen{u_2}$, as claimed.
\end{proof}

We now consider the intersections of the $H_i$. Visibly no two $H_i$ intersect in a subgroup containing an element of order $p$, so their intersections are all groups of order dividing $p-1$.

\begin{lem} The intersections $H_{13}$, $H_{14}$, $H_{23}$, $H_{24}$ are all of order $p-1$.
\end{lem}
\begin{proof} We simply give an element of order $p-1$ that is in the intersection of each pair of groups.
\begin{enumerate}
\item Of course $t\in H_1\cap H_4$, so $|H_1\cap H_4|=p-1$.
\item The element $h_3$ is in $H_3$ by definition and in $H_1$ as $h_3=u_1t$, so $|H_1\cap H_3|=p-1$. Similarly, the element $h_2$ lies in $H_2$ by definition and in $H_4$ as $h_2=tu_4$, so $|H_2\cap H_4|=p-1$.
\item For $H_2\cap H_3$, we claim it contains the element
\[ \begin{pmatrix} \zeta&1&1/(\zeta-1)\\0&1&-1\\0&0&\zeta^{-1}\end{pmatrix}.\]
This element is $u_2h_2$ and $h_3u_3^{-1}$, so $|H_2\cap H_3|=p-1$.
\end{enumerate}
This proves all of the claimed intersections.
\end{proof}

The last two intersections are $H_{12}$ and $H_{34}$.

\begin{lem} We have that $H_{34}=1$ and $|H_{12}|=2$, generated by\[ x=\begin{pmatrix}-1&1/(1-\zeta)&0\\0&1&0\\0&0&-1\end{pmatrix}.\]
\end{lem}
\begin{proof} To prove that $H_{34}=1$ is easy. All elements of $H_3$ have $(1,2)$-entry equal to $0$, so anything in the intersection $H_3\cap H_4$ must have this property. All elements of the normal subgroup $\gen{u_3}$ has $(1,2)$-entry equal to $0$, whereas $h_3$ does not; thus the only elements $h_3^iu_3^j$ that can lie in $H_3$ have $i=0$. Hence $H_3\cap H_4\leq \gen{u_3}$, but $u_3\notin H_4$ clearly, so the result holds.

For $H_1\cap H_2$, we claim that the intersection is generated by $x$, as in the statement. Suppose that some element of $H_1$ is equal to some element of $H_2$. Then we have that $t^iu_1^j=h_2^iu_2^k$ for some $i,j,k$. Then $u_1^j=t^{-i}h_2^iu_2^k$. We can compute $t^{-i}h_2^i$, so we obtain
\[ \begin{pmatrix} 1&0&0\\0&1&(\zeta^{-i}-1)/(\zeta^{-1}-1)
\\ 0&0&1\end{pmatrix}\begin{pmatrix}1&k&k\alpha-k(k-1)\zeta\\0&1&-2k\zeta\\0&0&1\end{pmatrix}=\begin{pmatrix}1&j&0\\0&1&0\\0&0&1\end{pmatrix},\]
where $\alpha=\zeta(\zeta-2)/(1-\zeta)$ is the $(1,3)$-entry of $u_2$. The top row of this yields $k=j$ and $k\alpha-k(k-1)\zeta=0$, which yields $k=0$ or $k=1/(\zeta-1)$. The case $k=0$ yields $t^i=h_2^i$, so $i=0$. The case $k=1/(\zeta-1)$ yields
\[ \begin{pmatrix} 1&0&0\\0&1&(\zeta^{-i}-1)/(\zeta^{-1}-1)
\\ 0&0&1\end{pmatrix}\begin{pmatrix}1&1/(\zeta-1)&0\\0&1&-2\zeta/(\zeta-1)\\0&0&1\end{pmatrix}=\begin{pmatrix}1&1/(\zeta-1)&0\\0&1&0\\0&0&1\end{pmatrix},\]
This yields $\zeta^-i=-1$, so $i=(p-1)/2$. This gives us an element of order $2$, and indeed if $m=(p-1)/2$ and $k=1/(\zeta-1)$, then $t^mu_1^k=h_2^mu_2^k$ is the claimed matrix $x$ above.
\end{proof}

Since $x\notin \gen{h_3}=H_1\cap H_3$ and $x\notin \gen{t}=H_1\cap H_4$, we see that $H_{123}=H_{124}=1$. We are now in a position to check the Ingleton inequality. The left-hand side is $p^2(p-1)^2$, and the right-hand side is $2(p-1)^4$. We have that $p^2\leq 2(p-1)^2$ whenever $p\geq 5$, so the quadruple is an Ingleton offender.

\medskip

We can define the subgroups above for an arbitrary $q$, not just $p$. The case $q=4$ yields a group of order $2^6\cdot 3=192$, and is \texttt{SmallGroup(192,1023)}. This will appear as an Ingleton violator in the next section. However, if we work with prime powers $q$ rather than primes $p$, the corresponding group $G$ is no longer supersoluble, so we restrict ourselves to the prime case.

%
%
%
%
%
%
%

\section{Ingleton offenders of small order}
\label{sec:smallorder}
In \cite{bostonnan2020} a number of examples of Ingleton violators and offenders in groups of small order are given, such as $A_4\times A_4$ and a group $2^4\rtimes F_{20}$, a split extension of an elementary abelian group of order $16$ by a Frobenius group of order $20$, but complete details of exactly what groups were found by Boston and Nan are not given there.

Here we provide a systematic study of Ingleton violators of small order. In particular, in Table \ref{tab:indingvios} we list all indomitable Ingleton violators of order at most $1023$ (the groups of order $1024$ are not publicly available at the time of writing). These were, of course, found with a computer. There are a few groups that are irreducible Ingleton violators but not indomitable Ingleton violators, and these are listed in Table \ref{tab:irringvios}. In the tables we provide the designation in the SmallGroups database, together with either their name (if they have one) or a description of their structure, the latter of which need not uniquely determine $G$ in complicated cases, but gives the reader an idea as to their structure. (We do not do this for the eight indomitable Ingleton violators of order $768$ because they have a very complicated structure and it is not enlightening.)

\begin{table}
\begin{center}
\begin{tabular}{ccc}
\hline $n$ & SmallGroup identifier & Standard name/description
\\\hline  $120$ & $34$ & $S_5$
\\ $144$ & $184$ & $A_4\times A_4$
\\ $192$ & $1023$ & Borel of $\PSL_3(4)$
\\ $288$ & $1024$, $1025$, $1026$ & $A_4\times S_4$, $A_4\wr 2$, $(A_4\times A_4).2$
\\ $320$ & $1635$ & $2^4\rtimes F_{20}$
\\ $324$ & $160$ & $3^3\rtimes A_4$
\\ $336$ & $208$ & $\PGL_2(7)$
\\ $360$ & $118$, $120$ & $A_6$, $(3\times A_5)\cdot 2$
\\ $384$ & $18135$, $18235$ & Borel of $\PSL_3(4).2\leq \mathrm{P\Gamma L}_3(4)$, $2^{2+4}\rtimes 6$
\\ $480$ & $218$ & $\GL_2(5)$
\\ $500$ & $23$ & $5^{1+2}_+\rtimes 4$
\\ $504$ & $156$, $157$ & $\PSL_2(8)$, $3\times \PSL_2(7)$
\\ $576$ & $5129$, $8652$, $8653$, $8654$, & $2^2.(A_4\times A_4)$, $(A_4\times A_4).4$, $S_4\times S_4$, $(A_4\times A_4).2^2$
\\ & $8655$, $8656$, $8657$ & $2\times A_4\times S_4$, $2\times (A_4\wr 2)$, $2\times (A_4\times A_4).2$
\\ $600$ & $151$ & $5^2\rtimes(S_3\times 4)$
\\ $640$ & $21536$ & $2\times (5_+^{1+2}\rtimes 4)$
\\ $648$ & $703$, $704$, $705$ & $3^3\rtimes S_4$, $3^3\rtimes S_4$, $3^3\rtimes (2\times A_4)$
\\ $660$ & $13$ & $\PSL_2(11)$
\\ $672$ & $1257$ & $2^5\rtimes F_{21}$
\\ $720$ & $763$ to $768$ & $S_6$, $\PGL_2(9)$, $M_{10}$, $A_6\times 2$, $A_4\times A_5$
\\ $768$ & $1084517$, $1084555$, $1085079$, $1085112$, & -
\\ & $1088651$, $1088660$, $1090131$, $1090134$ & -
\\ $864$ & $4675$, $4677$, $4679$& $(3\times A_4).2\times A_4$, $(3\times A_4^2).2$, $(3\times A_4^2).2$
\\ $882$ & $34$ & $7^2\rtimes(S_3\times 3)$
\\ $960$ & $11357$, $11358$, $11359$ & $2^4\rtimes A_5$, $2^4\rtimes A_5$, $2^4\rtimes( 15\rtimes 4)$
\\ $972$ & $877$ & $(3^3\rtimes A_4)\times 3$
\\ $1000$ & $91$ & $5^{1+2}_+\rtimes (2\times 4)\leq \text{ Borel of }\SL_3(5)$
\\ $1008$ & $882$ & $3\times \PGL_2(7)$
\\ \hline
\end{tabular}
\end{center}
\caption{All indomitable Ingleton violators of order up to $1023$}
\label{tab:indingvios}
\end{table}

\begin{table}
\begin{center}
\begin{tabular}{ccc}
\hline $n$ & SmallGroup identifier & Standard name/description
\\\hline $384$ & $5863$ & Borel of $2\cdot \PSL_3(4)$
\\ $576$ & $8277$, $8282$ & $2\cdot(A_4\wr C_2)$, $2\cdot(A_4\times A_4).2$
\\ $768$ & $1083945$ & Borel of $4_b.\PSL_3(4)$
\\ $864$ & $4680$ & $S_3\times A_4\times A_4$
\\ $960$ & $10871$ & $D_8\times S_5$
\\ \hline
\end{tabular}
\end{center}
\caption{All irreducible Ingleton violators of order up to $1023$ that are not indomitable Ingleton violators}
\label{tab:irringvios}
\end{table}

Many of these have large numbers of conjugacy classes of indomitable Ingleton offenders. We give the classes and scores of the offenders in the groups of order up to $504$ (above that it starts getting unwieldy to present them in tabular form) in Table \ref{tab:ingoffs}.

One example we referred to earlier was $3\times \PSL_2(7)$, which was interesting because it has a central subgroup and an Ingleton offender, but $\PSL_2(7)$ is not an Ingleton violator. We give the structure of one Ingleton offender now.

\begin{example} Let $G=3\times \PSL_2(7)$, generated by $(1,2,3)$ for the one factor, and $(6,9,10)(7,8,11)$ and $(4,11,5)(7,8,9)$ in the other. (These are the default generators given by Magma.) Let 
\[g=(1,2,3)(6,9,10)(7,8,11),\quad h=(1,2,3)(4,10,5)(6,9,7),\quad  x=(4,11)(5,10)(6,9)(7,8),\]
and let
\[ H_1=\gen{g,x},\quad H_2=\gen{h,x},\quad H_3=\gen{g,h^x},\quad H_4=\gen{h,g^x}\]
Then $H_1,H_2\cong A_4$, $H_3,H_4\cong F_{21}$, and the intersections $H_i\cap H_j$ have order $3$ except for $H_1\cap H_2=\gen x$ and $H_3\cap H_4=1$. The quadruple $(H_1,H_2,H_3,H_4)$ is an Ingleton offender, but the image modulo $Z(G)$ is not.
\end{example}

\begin{table}
\begin{center}
\begin{tabular}{cccc}
\hline Violator & No.\ of Offender classes & Ratios & Scores
\\ \hline $S_5$ & $1$ & $16/15$ & $0.01348$
\\ $A_4\times A_4$ & $1$ & $9/8$ & $0.02370$
\\ $\PSL_3(4)$ Borel & $48$ & $9/8$ & $0.02240$
\\ $A_4\times S_4$ & $1$ & $9/8$ & $0.02080$
\\ $A_4\wr C_2$ & $7$ & $9/8$ & $0.02080$ (6 times), $0.02370$
\\ $(A_4\times A_4).2$ & $7$ & $9/8$ & $0.02080$ (6 times), $0.02370$
\\ $2^4\rtimes F_{20}$ & $1$ & $32/25$ & $0.04280$
\\ $3^3\rtimes A_4$ & $2$ & $9/8$ & $0.02038$
\\ $\PGL_2(7)$ & $1$ & $8/7$ & $0.02295$
\\ $A_6$ & $32$ & $9/8$ (16 times)  & $0.02001$ (12 times), $0.02268$ (4 times)
\\ & & $16/15$ (16 times) & $0.01096$ (14 times), $0.01348$ (2 times)
\\ $(3\times A_5).2$ & $8$ & $16/15$ (6 times) & $0.01096$ (6 times)
\\ & & $6/5$ (2 times) & $0.03097$ (2 times)
\\ $\PSL_3(4).2$ Borel &$56$ & $9/8$ & $0.01979$ (52 times), $0.02240$ (2 times)
\\ $2^{2+4}\rtimes 6$ & $16$ & $9/8$ & $0.01979$
\\ $\GL_2(5)$ & 4 & $16/15$ & $0.01045$
\\ $5_+^{1+2}\rtimes 4$ & $60$ & $32/25$ & $0.03972$
\\ $\PSL_2(8)$ & $1$ & $7/6$ & $0.02477$
\\ $3\times \PSL_2(7)$ & $2$ & $9/8$ & $0.01892$
\\ \hline
\end{tabular}
\end{center}
\caption{Ingleton offenders and scores for the Ingleton violators in Table \ref{tab:indingvios} of order at most $504$}
\label{tab:ingoffs}
\end{table}

Looking at Table \ref{tab:indingvios}, one has two obvious questions:

\begin{question}
\begin{enumerate}
\item Are there any nilpotent Ingleton violators?
\item Are there any Ingleton violators of odd order?
\end{enumerate}
\end{question}

In both situations, since the collection of groups is closed under subgroups and quotients, any Ingleton violator of smallest order possesses only indomitable Ingleton offenders. Call an Ingleton violator \emph{minimal} if no proper subgroup contains an Ingleton offender. When searching with a computer, we can obtain stronger restrictions on minimal Ingleton violators.

\begin{prop} If $\mathbf H=(H_1,H_2,H_3,H_4)$ is an Ingleton offender in a minimal Ingleton violator $G$ then $G=\gen{H_i,H_j}$ for any $i\neq j$. Furthermore, if $\{i,j\}\neq \{3,4\}$, then for all $k\neq i,j$, $G=\gen{H_k,H_{ij}}$.
\end{prop}
\begin{proof} Let $M$ be any maximal subgroup of $G$. By assumption, $M$ is not an Ingleton violator, and so the $\bar H_i=H_i\cap M$ cannot be an Ingleton offender. Notice that if $H_i\leq M$ then $\bar H_I=H_I$ if $i\in I$. Thus when comparing the Ingleton inequalities for the $H_i$ and the $\bar H_i$, the right-hand side must be have reduced by more than the left-hand side in order for the inequality to fail for the $H_i$ and pass for the $\bar H_i$.

If $H_1,H_2\leq M$ then the right-hand side of the inequality remains unchanged, so this contradiction proves that $G=\gen{H_1,H_2}$ already. If $H_3,H_4\leq G$ then only $H_{12}$ can be reduced from the right-hand side. But $|H_{12}:\bar H_{12}|=|H_{12}:H_{12}\cap M|\leq |H_1:H_1\cap M|$ by Lemma \ref{lem:indexreduction}, and so the right-hand side reduces by at most the left-hand side, leading to a contradiction.

If $H_1,H_3\leq M$ then, similarly to the previous case, only $H_{24}$ can reduce from the right-hand side, and by at most that of $H_2$ on the left-hand side. This proves the generation by $H_i$ and $H_j$.

For the stronger statement, first assume that $\{i,j\}=\{1,2\}$. Let $M$ be a maximal subgroup containing $H_3$ and $H_{12}$, and intersect all the $H_i$ with $M$. This cannot reduce $H_{12}$, $H_{13}$ and $H_{23}$, but could reduce $H_{14}$ (but by at most that of $H_1$) and $H_{24}$ (but by at most that of $H_2$). This proves the result for $H_{12}$.

If $\{i,j\}=\{1,3\}$, then first let $M$ be a maximal subgroup of $G$ containing $H_2$ and $H_{13}$. Similarly to above, the only term on the right-hand side that can be reduced is $H_{14}$, by at most that of $H_1$. Thus the result holds here. If $M$ is a maximal subgroup containing $H_4$ and $H_{13}$ then we can get a reduction in $H_{12}$ and $H_{23}$, but not by more than $H_1$ and $H_2$ respectively.
\end{proof}

Because there are very few minimal Ingleton violators, as a proportion of the groups of a given order, when looking for arbitrary Ingleton violators one can proceed inductively, and can assume that a group $G$ is either a minimal Ingleton violator or that it has a subgroup that is a minimal Ingleton violator, and that these are all known by induction. This allows us quickly prove (with a computer) that only a couple of the 1090235 groups of order $768$ can be a minimal Ingleton violator, and so we can concentrate computational resources on the few dozen groups that have Ingleton violator proper subgroups.

There are very few Ingleton violators in general: of the 11759892 groups of order at most $1023$, only $391$ groups are Ingleton violators, and $61$ are irreducible Ingleton violators. Even if one excludes the nilpotent groups from consideration, there are still 1137397 groups of order at most $1023$, so only 0.03\% of non-nilpotent groups of order at most $1023$ are Ingleton violators, and 0.005\% are irreducible Ingleton violators.

Although we proved in Corollary \ref{cor:H1notcyclic} that $H_1$ and $H_2$ cannot be cyclic, and we remarked that $H_3$ and $H_4$ cannot have prime-power order, in general we do not know if $H_3$ and $H_4$ could be cyclic. Certainly $H_3$ and $H_4$ can be abelian, as in the $A_4\times A_4$ case, but we do not know if $H_1$ and $H_2$ can be abelian. We summarize the open questions about Ingleton offenders we have suggested here now.
\begin{enumerate}
\item Are there any nilpotent Ingleton violators?
\item Are there any Ingleton violators of odd order?
\item Are there any Ingleton offenders with $H_1$ abelian?
\item Are there any Ingleton offenders with $H_3$ cyclic?
\item Can all $|H_{ij}|$ have odd order?
\end{enumerate}

\end{document}